\documentclass[12pt]{amsart}

\usepackage{tikz}
\usetikzlibrary{decorations.pathreplacing}
\usetikzlibrary{shapes,positioning}

\usepackage{cite}

\usepackage[margin=1in]{geometry}
\usepackage{graphicx}
\usepackage{amsthm, amsmath, amssymb, tikz, bbm,amsfonts}
\usepackage{mathtools}
\usepackage{xifthen}
\usepackage{hyperref}
\usepackage{comment}

\usepackage[shortlabels]{enumitem}
\usepackage{todonotes}
\usetikzlibrary{calc,shapes, backgrounds}

\newcommand{\op}{\operatorname}

\newcommand{\diam}{\op{sdiam}}
\newcommand{\rad}{\operatorname{srad}}

\newcommand{\e}{\op{e}}

\newcommand{\old}[1]{{}}

\newtheorem{thm}{Theorem}

\newtheorem{cor}[thm]{Corollary}
\newtheorem{lem}[thm]{Lemma}
\newtheorem{defn}[thm]{Definition}

\newtheorem{claim}[thm]{Claim}

\newtheorem{prop}[thm]{Proposition}
\newtheorem{obs}[thm]{Observation}
\newtheorem{conj}[thm]{Conjecture}

\title{The Steiner $k$-radius and Steiner $k$-diameter of connected graphs for $k\geq 4$}

\author{Josiah Reiswig}

\address{Josiah Reiswig\\ Department of Mathematics \\ Anderson University
\\ Anderson SC 29621 \\ USA}
\email{jreiswig@andersonuniversity.edu}

\thanks{
The author was also supported in part by the National
Science Foundation contract DMS-1600811.
}

\begin{document}

\begin{abstract}
Given a connected graph $G=(V,E)$ and a vertex set $S\subset V$, the {\em 
Steiner distance} $d(S)$ of $S$ is the size of a minimum spanning tree of $S$ in 
$G$. For a connected graph $G$ of order $n$ and an integer $k$ with $2\leq k 
\leq n$, the $k$-eccentricity of a vertex $v$ in $G$ is the maximum value 
of $d(S)$ over all $S\subset V$ with $|S|=k$ and $v\in S$. The minimum 
$k$-eccentricity, $\rad_k(G)$, is called the $k$-radius of $G$ while the maximum 
$k$-eccentricity, $\diam_k(G)$, is called the $k$-diameter of $G$. In 1990, 
Henning, Oellermann, and Swart
[\textit{Ars Combinatoria} \textbf{12} 13-19, (1990)]
showed that there exists a graph $H_k$ such that 
$\diam_k(H_k) = \frac{2(k+1)}{2k-1}\rad_k(H_k)$. The 
authors also conjectured that for any $k\geq 2$ and connected graph $G$
$\diam_k(G) \leq \frac{2(k+1)}{2k-1}\rad_k(G)$. The authors
provided proofs of the conjecture for $k=3$ and $4$. Their proof for $k=4$,
however, was incomplete. In this note, 
we disprove the conjecture for $k\geq 5$ by proving that the bound 
$\diam_k(G)\leq \frac{k+3}{k+1}\rad_k(G)$ is tight for $k\geq 5$.
We then provide a complete proof for $k=4$ and
identify the error in the previous proof of this case.
\end{abstract}

\maketitle

\section{Introduction and Notation}

Given a graph $G=(V,E)$ with vertex set $V=V(G)$ and edge set $E=E(G)$, we let 
$|G|=|V(G)|$ denote the order of $G$ and $\|G\|=|E(G)|$ denote the size of $G$. 
The distance in $G$ between two vertices $u,v\in V$, denoted $d_G(u,v)$, is the 
length of the shortest path in $G$ between $u$ and $v$. If there is no path 
between $u$ and $v$, we say that $d_G(u,v)=\infty$. The \textit{eccentricity} of 
a vertex $v$ in $G$ is defined as $\e(v):=\max\{d_G(u,v):u\in V(G)\}$. The 
\textit{radius} $\op{rad}(G)$ is defined as $\min\{\e(v):v\in V(G)\}$ and the 
\textit{diameter} of $G$ $\op{diam}(G)$ is defined as $\max\{\e(v):v\in V(G)\}$. The center of 
$G$, denoted $C(G)$, is the subset of vertices $v\in G$
such that $e(v)=\rad(G)$. 
If $H$ is a subgraph of $G$ and $v\in V(G)$, then \textit{the distance from $v$ 
to $H$}, denoted $d_G(v,H)$, is defined as $\min\{d_G(v,u):u\in V(H)\}$. The \textit{neighborhood} of a vertex $v$, is defined as $N_G(v) := \{u:uv\in E(G)\}$.

The distance between two vertices $v$ and $u$ can be viewed as the minimal size 
of a connected subgraph (in this case, a path) of $G$ containing $v$ and $u$. This suggests a 
generalization of distance. Introduced in \cite{CHA89}, the \textit{Steiner 
distance} in $G$ of a non-empty set $S\subset V(G)$, denoted $d_G(S)$, is defined 
as the size of the smallest connected subgraph of $G$ containing all elements of 
$S$. Necessarily, such a minimum subgraph must be a tree.
When the context is clear, we simply write $d_G(S)$ as $d(S)$.

Given an integer $k\geq 2$, the \textit{Steiner $k$-eccentricity} of a vertex 
$v$ in $G$, denoted $\e_k(v)$, is defined as the maximum Steiner distance of all 
vertex subsets of $G$ of size $k$ containing $v$. More succinctly,
$
\e_k(v)=\max_{S\subset V(G),|S|=k}\{d(S):v\in S\}.
$
The \textit{Steiner $k$-radius}, denoted $\rad_k(G)$, is then defined as 
$
\rad_k(G):=\min\{e_k(v):v\in G\},
$
while the \textit{Steiner $k$-diameter}, denoted $\diam_k(G)$ is then defined as 
$
\diam_k(G):=\max\{e_k(v):v\in G\}.
$
The Steiner $k$-center, $C_k(G),$ is the subgraph 
induced by all vertices $v$ with $e_k(v)=\rad_k(G)$.
For a general connected graph, 
the following connection between the Steiner distance and the standard
distance is immediate.

\begin{obs}\label{obs}
If $G$ is a connected graph and $v\in V(G)$, then $\e_2(v)=\e(v)$, 
$\rad_2(G)=\op{rad}(G)$, $\diam_2(G)=\op{diam}(G)$, and $C_2(G)=C(G)$. 
\end{obs}

Much research has been focused on bounding the Steiner $k$-diameter by
other parameters.
Danklemann, Swart, and
Oellermann in \cite{DAN99}
gave upper bounds for $\diam_k(G)$ with respect to
the minimum degree and complement of $G$. 
In \cite{ALI13}, Ali gave an upper bounds for
$\diam_k(G)$ with respect to the girth of $G$. In this paper, we
will bound $\diam_k(G)$ with respect to $\rad_k(G)$. 

In light of Observation~\ref{obs}, it is well known that 
$\diam_2(G)\leq 2\rad_2(G).$
In their paper introducing the Steiner distance, the authors of 
\cite{CHA89} generalized this result by showing
that for any tree $T$,
$$
\diam_k(T)\leq \frac{k}{k-1}\rad_k(T).
$$
The authors conjectured that this result extended to all connected graphs. In 1990,
however, Henning, Oellermann, and Swart \cite{HEN90} showed via construction
that for each $k\geq 2$, there exists a graph $G^*$ satisfying 
$
\diam_k(G^*) = \frac{2(k+1)}{2k-1}\rad_k(G^*).
$
Furthermore, they conjectured that this gap was largest possible.
\begin{conj}[See \cite{HEN90}] \label{bad}
Suppose that $G$ is a connected graph with order at least $k$. Then 
$$
\diam_k(G)\leq \frac{2(k+1)}{2k-1}\rad_k(G).
$$
\end{conj}
In the same paper, proofs of the conjecture were provided for $k=3,4$.
The proof for
$k=4$, however, was incorrect.

We break this writing into several divisions. In Section~\ref{def}, we 
make necessary definitions and prove some preliminary lemmas 
required for our 
main results. In Section~\ref{lb}, we prove our main result:
\begin{thm}
\label{MAIN}
If $G$ is a connected graph and $k\geq 5$ is an integer, then 
$$
\diam_k(G)\leq \frac{k+3}{k+1}\rad_k(G). 
$$
\end{thm}
In Section~\ref{ub} we show that this bound is tight. 
In Section~\ref{ub4}, we provide a correct proof to confirm the 
conjecture in \cite{HEN90} for $k=4$ building on the portions proven
in \cite{HEN90}. Following this, in Section~\ref{cor},
we identify the error in the proof
of Conjecture~\ref{bad} for $k=4$ provided in \cite{HEN90}.
Finally, Section~\ref{complete} provides the correct arguments 
originally presented in \cite{HEN90} for the 
case $k=4$ to complete the proof. 
To summarize the results of this paper and related results,
Table~\ref{table} gives the maximum value of the
ratio $\diam_k(G)/\rad_k(G)$ for a connected graph $G$ as prescribed 
by \cite{HEN90} and Theorem~\ref{MAIN}. 

\begin{table}[!h]
\centering
\caption{Values of $\diam_k(G)/\rad_k(G)$ as found in \cite{HEN90} and this paper.}
\begin{tabular}{|c|c|c|}\hline
$k$ & ${\diam_k(G)}/{\rad_k(G)}$ & Reference\\\hline
$3$ & ${8}/{5}$ & \cite{HEN90}\\
$4$ & ${10}/{7}$ & \cite{HEN90} and Section~\ref{ub4}\\
$\geq 5$ & $(k+3)/(k+1)$ & Section~\ref{lb}\\\hline
\end{tabular}
\label{table}
\end{table}

\section{Definitions and preliminary lemmas}
\label{def}
Let $k\geq 2$ be a positive integer and suppose that $G$ is a connected graph of 
order at least $k$. Then there exists a set $D=\{v_1,v_2,\ldots,v_k\}\subset V(G)$ 
such that ${d(D)=\diam_k(G)}$. Similarly, there exists $v_0 \in V(G)$ satisfying 
$\e_k(v_0)=\rad_k(G)$.  We may now make the following definitions, which closely 
follow definitions made in \cite{HEN90}.

\begin{defn}
\label{defn}
Suppose that $G$ is a connected graph of order at least $k$. Assume that 
$D=\{v_1,v_2,\ldots,v_k\}$ with $d(D)=\diam_k(G)$ and $\e_k(v_0)=\rad_k(G)$.
For each $1\leq i\leq k$,
\begin{enumerate}
\item Define $D_i:=(D \setminus\{v_i\})\cup\{v_0\}$;

\item Define $T_i$ to be a Steiner tree for $D_i$;

\item Define $T_i'$ to be the smallest subtree of $T_i$ containing
$D_i\setminus \{v_0\}$;

\item Define $\ell_i:=\|T_i\|-\|T_i'\|$. Without loss of generality, we assume 
that $\ell_1\leq \ell_j$ for $j\geq 2$. 
\end{enumerate}
\end{defn}
Of course, if $v_0\in D$, we have that $\rad_k(G)=\diam_k(G)$. So if 
$\rad_k(G)<\diam_k(G)$ we must have $v_0\notin D$.
It is worth noting that $v_i$ is the only element of $D\cup \{v_0\}$ not 
necessarily contained in the tree $T_i$, while the tree $T_{i}'$ need not 
contain $v_0$. Figure~\ref{T1andT1'} illustrates the difference between the 
trees $T_1$ and $T_1'$ for $k=3$.

\begin{figure}[ht!]
\centering
\begin{tabular}{c|c}
\begin{tikzpicture}
node/.style={draw=black,thick,circle,inner sep=0pt,fill=black}]
\newcommand\X{3};
\newcommand\Y{1};
\node[minimum size=1mm,inner sep=0pt] (v2) at (0,0) [circle, 
fill=black,label=below:$v_2$] {};

\node[minimum size=1mm,inner sep=0pt] (v0) at (0,2*\Y) [circle, 
fill=black,label=above:$v_0$] {};
\node[minimum size=1mm,inner sep=0pt] (s) at (1*\X,1*\Y) [circle, fill=black] 
{};
\node[minimum size=1mm,inner sep=0pt] (v3) at (2*\X,1*\Y) [circle, 
fill=black,label=right:$v_3$] {};
\node[minimum size=1mm,inner sep=0pt] (v1) at (1.5*\X,0*\Y) [circle, 
fill=black,label=below:$v_1$] {};
\draw  (s) edge (v3) (s) edge (v0);
\draw[decorate, decoration={brace}] ([xshift = 1ex,yshift=0ex]v0.center) -- 
([xshift = -.5ex,yshift=.5ex]s.center) node[midway, above right] {$\ell_1$};
\draw (v2) edge (s);
\node[minimum size=1mm,inner sep=0pt] (a) at (1.5*\X,2*\Y) {The tree $T_1$};
\end{tikzpicture}
&
\begin{tikzpicture}
node/.style={draw=black,thick,circle,inner sep=0pt,fill=black}]
\newcommand\X{3};
\newcommand\Y{1};
\node[minimum size=1mm,inner sep=0pt] (v2) at (0,0) [circle, 
fill=black,label=below:$v_2$] {};

\node[minimum size=1mm,inner sep=0pt] (v0) at (0,2*\Y) [circle, 
fill=black,label=above:$v_0$] {};
\node[minimum size=1mm,inner sep=0pt] (s) at (1*\X,1*\Y) [circle, fill=black] 
{};
\node[minimum size=1mm,inner sep=0pt] (v3) at (2*\X,1*\Y) [circle, 
fill=black,label=right:$v_3$] {};
\node[minimum size=1mm,inner sep=0pt] (v1) at (1.5*\X,0*\Y) [circle, 
fill=black,label=below:$v_1$] {};
\draw  (s) edge (v3);
\draw (v2) edge (s);
\node[minimum size=1mm,inner sep=0pt] (a) at (1.5*\X,2*\Y) {The tree $T_1'$};
\end{tikzpicture}
\end{tabular}
\caption{Possible trees $T_1$ and $T_1'$ for $k=3$. Vertices of degree 2 are 
not drawn.}
\label{T1andT1'}
\end{figure}
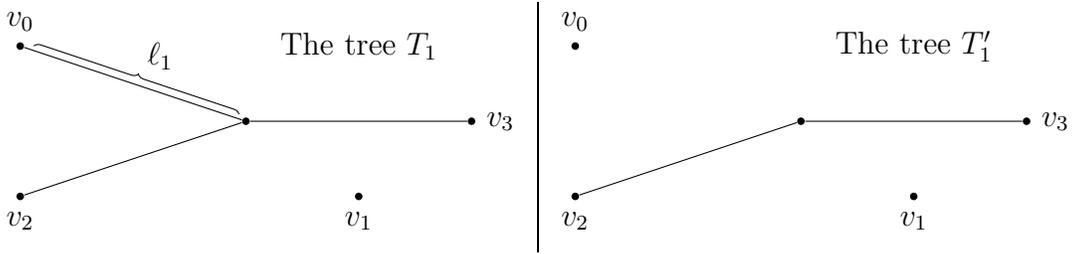

Note that $\ell_i=d_{T_i}(v_0,T_i')$. From Definition~\ref{defn}, we make the following observation.

\begin{obs}
\label{kobs}
Suppose that $k\geq 2$ is an integer and that $G$ is a connected graph with at 
least $k$ vertices. Let $\ell_i$, $T_i$, and $T_i'$ be defined as in Definition~\ref{defn}.
If $\diam_k(G)>p\rad_k(G)$ for some $p>0$,  then for each $1\leq 
i\leq k$, we have the following:
\begin{enumerate}
\item $\displaystyle \|T_i\| \leq \rad_k(G)<\frac{1}{p}\diam_k(G)$, and
\item $\displaystyle \|T_i'\|= \|T_i\| -\ell_i < \frac{1}{p}\diam_k(G)-\ell_1$.
\end{enumerate}
\end{obs}

With Observation~\ref{kobs} in mind, we now prove our first lemma.

\begin{lem}
\label{distlem}
Suppose that $G$ is a connected graph of order $n\geq k$. Let $\ell_i$, 
$T_i$, and $T_i'$ be defined as in Definition~\ref{defn}. If 
$\diam_k(G)>p\cdot \rad_k(G)$ with $p>1$, then for $1< i , j \leq k$ with $i\neq j$, 
the following hold:
\begin{enumerate}
\item $\displaystyle d_{T_1}(v_i,v_0)>\frac{p-1}{p}\diam_k(G),$ and

\item $\displaystyle d_{T_1}(v_i,v_j)>\frac{p-1}{p}\diam_k(G)+\ell_1.$
\end{enumerate} 
\end{lem}

\begin{proof}
For the first inequality, note that adjoining the tree $T_i$ with the path in 
$T_1$ between $v_i$ and $v_0$ generates a connected subgraph of $G$ spanning 
$D$. Hence, 
$$
\|T_i\|+d_{T_1}(v_i,v_0)\geq \diam_k(G),
$$
which implies that 
$$
d_{T_1}(v_i,v_0)\geq \diam_k(G)-\|T_i\|.
$$
In view of Observation~\ref{kobs}, we see that 
\begin{align*}
d_{T_1}(v_i,v_0)
&> \diam_k(G)-\left( \frac{1}{p}\diam_k(G) \right)\\
&=\frac{p-1}{p}\diam_k(G). 
\end{align*}

For the second inequality, we similarly note that adjoining the tree $T_i'$ with the path in 
$T_1$ between $v_i$ and $v_j$ generates a connected subgraph of $G$ spanning 
$D$. Hence, 
$$
\|T_i'\|+d_{T_1}(v_i,v_j)\geq \diam_k(G),
$$
which implies that 
$$
d_{T_1}(v_i,v_j)\geq \diam_k(G)-\|T_i'\|.
$$
Applying Observation~\ref{kobs} a second time, we have that
\begin{align*}
d_{T_1}(v_i,v_j)
&>\diam_4(G)-\left(  \frac{1}{p}\diam_k(G)-\ell_i \right)\\
&= \frac{p-1}{p}\diam_k(G)+\ell_i\\
&\geq  \frac{p-1}{p}\diam_k(G)+\ell_1.
\end{align*}
\label{obsdist}
\end{proof}

With Lemma~\ref{distlem} in hand, we derive the following corollary. 

\begin{cor}
\label{cor}
Using the definitions and notations provided in Definition~\ref{defn}, if 
$1<i\neq j\leq k$ and
$$
\diam_k(G)>\frac{10}{7}\rad_k(G),
$$
then
\begin{enumerate}
\item $d_{T_1}(v_i,v_0)>\frac{3}{10}\diam_k(G)$, and

\item $d_{T_1}(v_i,v_j)>\frac{3}{10}\diam_k(G)+\ell_1$.
\end{enumerate}
Furthermore, if $1<i\neq j\leq k$ and 
$$
\diam_k(G)>\frac{k+3}{k+1}\rad_k(G),
$$
then 
\begin{enumerate}
\item $d_{T_1}(v_i,v_0)>\frac{2}{k+3}\diam_k(G)$, and

\item $d_{T_1}(v_i,v_j)>\frac{2}{k+3}\diam_k(G)+\ell_1$.
\end{enumerate}
\label{distlemobs}
\end{cor}
With these definitions and results in hand, we are prepared to prove our main 
result.

\section{Proof of Theorem \ref{MAIN}}\label{lb}

\begin{proof}
Suppose towards a contradiction that there exists a connected 
graph $G$ such that 
$$
\diam_k(G) > \frac{k+3}{k+1}\rad_k(G).
$$
This implies that 
\begin{equation}
\label{krad}
\rad_k(G)<\frac{k+1}{k+3}\diam_k(G). 
\end{equation}
Suppose $D=\{v_1,v_2,\ldots,v_k\}$ is a set of $k$ vertices such that 
$d(D)=\diam_k(G)$. Let $v_0\in C_k(G)$. For $1\leq i\leq k$, define $D_i$, 
$T_i$, $T_i'$, and $\ell_i$ as in Definition~\ref{defn}. Again, we assume that 
$\ell_{1}\leq \ell_j$ for $j\geq 2$. We have that $\|T_1\|\leq \rad_k(G)$. Let 
$x$ be the vertex in $T_1'$, which is closest to $v_0$ in 
$T_1$. It is possible that $x=v_0$. We now root $T_1$ at $v_0$ and 
consider the following two cases. 

\noindent \textbf{Case 1:} $x\in D_i\setminus\{v_0\}=\{v_2,\ldots, v_k\}$.

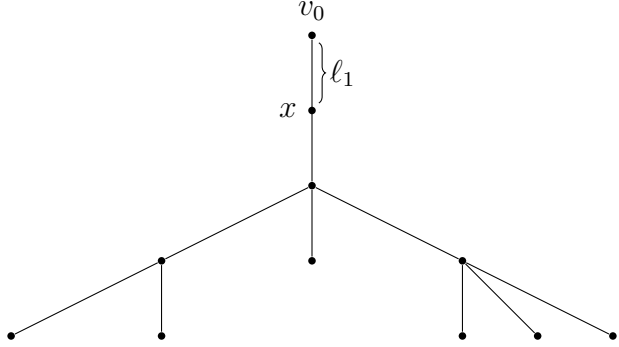
\begin{figure}[ht!]
\centering
\begin{tabular}{c}
\begin{tikzpicture}
\newcommand\X{2};
\newcommand\Y{1};
\node[minimum size=1mm,inner sep=0pt] (r) at (0,1*\Y) [circle, 
fill=black,label=above:$v_0$] {};
\node[minimum size=1mm,inner sep=0pt] (s) at (0*\X,-1*\Y) [circle, 
fill=black] {};
\node[minimum size=1mm,inner sep=0pt] (a) at (0*\X,0*\Y) [circle, 
fill=black,label=left:$x$] {};
\node[minimum size=1mm,inner sep=0pt] (i1) at (-1*\X,-2*\Y) [circle, 
fill=black,] {};
\node[minimum size=1mm,inner sep=0pt] (i2) at (1*\X,-2*\Y) [circle, fill=black,] 
{};
\node[minimum size=1mm,inner sep=0pt] (u1) at (-2*\X,-3*\Y) [circle, 
fill=black,] {};
\node[minimum size=1mm,inner sep=0pt] (u2) at (-1*\X,-3*\Y) [circle, fill=black] 
{};
\node[minimum size=1mm,inner sep=0pt] (u3) at (1*\X,-3*\Y) [circle, fill=black,] 
{};
\node[minimum size=1mm,inner sep=0pt] (u4) at (2*\X,-3*\Y) [circle, fill=black] 
{};
\node[minimum size=1mm,inner sep=0pt] (r1) at (0*\X,-2*\Y) [circle, fill=black] 
{};
\node[minimum size=1mm,inner sep=0pt] (r2) at (1.5*\X,-3*\Y) [circle, 
fill=black] {};

\draw[decorate, decoration={brace}, yshift=0ex,xshift=.5ex]  (0*\X,1*\Y-.1*\Y) 
-- node[right] {$\ell_1$}  (0*\X,0*\Y+.1*\Y);

\draw (r) -- (s) (s) -- (i1)  (i1)--(u1) (i1)--(u2) (s) -- (i2)   (i2) -- (u3) 
(i2) -- (u4) (s) -- (r1) (r2)--(i2);
\end{tikzpicture}
\end{tabular}
\caption{A possible picture of the tree $T_1$ in case 1. Unnamed vertices of degree 2 are not drawn.}
\label{case1}
\end{figure}

Since $x\in D_i$, we have that $d_{T_1}(v_0,x)=d_{T_1}(v_0,v_i)$ for some $2\leq 
i\leq k$. 
Then, by Corollary~\ref{distlemobs}, we have $\ell_1>\frac{2}{k+3}\diam_k(G)$. 
Traversing $T_1$ via a 
depth first search and returning to $v_0$ induces a new labeling of the elements of 
$D_1$ in the following way: Let $u_1,u_2,\ldots,u_{k-1}$ be a relabeling of the
vertices $v_2,\ldots,v_k$ in the order in which these vertices are visited first
in the depth first search. By Corollary~\ref{distlemobs}, we have that
$d_{T_1}(v_0,u_1)>\frac{2}{k+3}\diam_k(G)$ and 
$d_{T_1}(v_0,u_{k-1})>\frac{2}{k+3}\diam_k(G)$. Furthermore, 
Corollary~\ref{distlemobs} asserts that
$d_{T_1}(u_i,u_j)>\frac{2}{k+3}\diam_k(G)+\ell_1$. Since 
$\ell_1>\frac{2}{k+3}\diam_k(G)$, the length of this traversal is greater than 
\begin{align*}
&2\cdot \frac{2}{k+3}\diam_k(G)+(k-2)\left(\frac{2}{k+3}+\ell_1\right)\\
&>2\cdot \frac{2}{k+3}\diam_k(G)+(k-2)\left(\frac{2}{k+3}+\frac{2}{k+3}\right)\\
&=\frac{4(k-1)}{k+3}\diam_k(G).
\end{align*}
This traversal also visits each edge of $T_1$ exactly twice, which implies that 
$$
2\rad_k(G)\geq 2\|T_1\|>\frac{4(k-1)}{k+3}\diam_k(G).
$$
Since $k\geq 5$, we have contradicted equation \eqref{krad}. 

\textbf{Case 2:} $x\notin D_i \setminus \{v_0\}$. 

Since $x\notin D_i \setminus \{v_0\}$, we have 
that $x$ has at least 2 children. Pick a child of $x$, say $c$. 
Let $H_1$ be the tree induced by vertices of the $v_0c$ path and descendants of $c$, and let $H_2$ be the tree
obtained from $T_1$ by removing $c$ and its descendants.
Figure~\ref{HTREES} illustrates the 
differences between $T_1,$ $H_1$, and $H_2$. 

\begin{figure}[ht!]
\centering
\begin{tabular}{ccc}
The Tree $T_1$ & The Tree $H_1$ & The Tree $H_2$\\
\begin{tikzpicture}
\newcommand\X{1};
\newcommand\Y{1};
\node[minimum size=1mm,inner sep=0pt] (r) at (0,0) [circle, 
fill=black,label=above:$v_0$] {};
\node[minimum size=1mm,inner sep=0pt] (s) at (0*\X,-1*\Y) [circle, 
fill=black,label=left:$x$] {};
\node[minimum size=1mm,inner sep=0pt] (i1) at (-1*\X,-2*\Y) [circle, 
fill=black,label=above left:$c$] {};
\node[minimum size=1mm,inner sep=0pt] (i2) at (1*\X,-2*\Y) [circle, fill=black,] 
{};
\node[minimum size=1mm,inner sep=0pt] (u1) at (-2*\X,-3*\Y) [circle, 
fill=black,] {};
\node[minimum size=1mm,inner sep=0pt] (u2) at (-1*\X,-3*\Y) [circle, fill=black] 
{};
\node[minimum size=1mm,inner sep=0pt] (u3) at (1*\X,-3*\Y) [circle, fill=black,] 
{};
\node[minimum size=1mm,inner sep=0pt] (u4) at (2*\X,-3*\Y) [circle, fill=black] 
{};
\node[minimum size=1mm,inner sep=0pt] (r1) at (0*\X,-2*\Y) [circle, fill=black] 
{};
\node[minimum size=1mm,inner sep=0pt] (r2) at (1.5*\X,-3*\Y) [circle, 
fill=black] {};

\draw[decorate, decoration={brace}, yshift=0ex,xshift=.5ex]  (0*\X,0*\Y-.1*\Y) 
-- node[right] {$\ell_1$}  (0*\X,-1*\Y+.1*\Y);

\draw (r) -- (s) (s) -- (i1)  (i1)--(u1) (i1)--(u2) (s) -- (i2)   (i2) -- (u3) 
(i2) -- (u4) (s) -- (r1) (r2)--(i2);
\end{tikzpicture}
\hspace{.3in}
&
\hspace{.3in}
\begin{tikzpicture}
\newcommand\X{1};
\newcommand\Y{1};
\node[minimum size=1mm,inner sep=0pt] (r) at (0,0) [circle, 
fill=black,label=above:$v_0$] {};
\node[minimum size=1mm,inner sep=0pt] (s) at (0*\X,-1*\Y) [circle, 
fill=black,label=left:$x$] {};
\node[minimum size=1mm,inner sep=0pt] (i1) at (-1*\X,-2*\Y) [circle, 
fill=black,label=above left:$c$] {};
\node[minimum size=1mm,inner sep=0pt] (u1) at (-2*\X,-3*\Y) [circle, fill=black] 
{};
\node[minimum size=1mm,inner sep=0pt] (u2) at (-1*\X,-3*\Y) [circle, fill=black] 
{};

\draw[decorate, decoration={brace}, yshift=0ex,xshift=.5ex]  (0*\X,0*\Y-.1*\Y) 
-- node[right] {$\ell_1$}  (0*\X,-1*\Y+.1*\Y);

\draw (r) -- (s) (s) -- (i1)  (i1)--(u1) (i1)--(u2);
\end{tikzpicture}
\hspace{.3in}
&
\hspace{.3in}
\begin{tikzpicture}
\newcommand\X{1};
\newcommand\Y{1};
\node[minimum size=1mm,inner sep=0pt] (r) at (0,0) [circle, 
fill=black,label=above:$v_0$] {};
\node[minimum size=1mm,inner sep=0pt] (s) at (0*\X,-1*\Y) [circle, 
fill=black,label=left:$x$] {};
\node[minimum size=1mm,inner sep=0pt] (i2) at (1*\X,-2*\Y) [circle, fill=black,] 
{};
\node[minimum size=1mm,inner sep=0pt] (u3) at (1*\X,-3*\Y) [circle, fill=black,] 
{};
\node[minimum size=1mm,inner sep=0pt] (u4) at (2*\X,-3*\Y) [circle, fill=black] 
{};
\node[minimum size=1mm,inner sep=0pt] (r1) at (0*\X,-2*\Y) [circle, fill=black] 
{};
\node[minimum size=1mm,inner sep=0pt] (r2) at (1.5*\X,-3*\Y) [circle, 
fill=black] {};

\draw[decorate, decoration={brace}, yshift=0ex,xshift=.5ex]  (0*\X,0*\Y-.1*\Y) 
-- node[right] {$\ell_1$}  (0*\X,-1*\Y+.1*\Y);

\draw (r) -- (s) (s) -- (i2)   (i2) -- (u3) (i2) -- (u4) (s) -- (r1) (r2)--(i2);
\end{tikzpicture}
\end{tabular}
\caption{A possible picture of $T_1$, $H_1$, and $H_2$ as in case 2. Unnamed vertices of 
degree 2 are not drawn.}
\label{HTREES}
\end{figure}
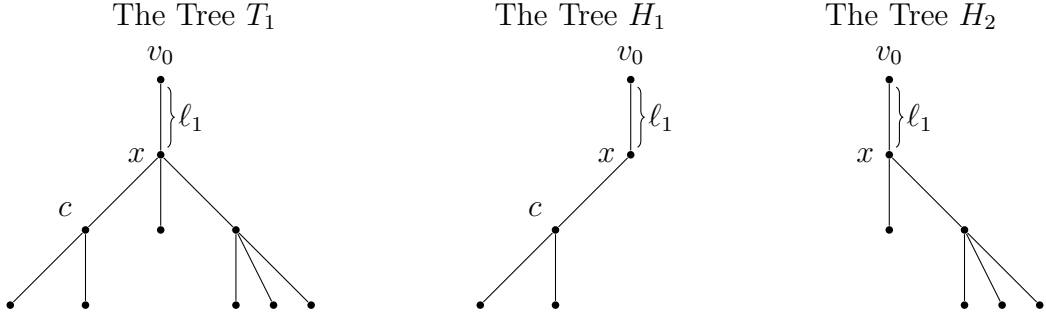

Both $H_1$ and $H_2$ contain elements of $D_i$. We observe that $E(H_1)\cup 
E(H_2)=E(T_1)$ while the intersection of $E(H_1)$ and $E(H_2)$ is
the path in $T_1$ 
between $v_0$ and $x$. Hence,
\begin{equation}
\label{heqn}
\|H_1\|+\|H_2\|-\ell_1=\|T_1\|<\frac{k+1}{k+3}\diam_k(G).
\end{equation}
It is easy to see that $|V(H_1)\cap D_1|+|V(H_2)\cap D_1|=k+1$ since $v_0$
(and only $v_0$) is 
included in both subtrees. As in the previous case, we root $H_1$ and $H_2$ at 
$v_0$ and perform a depth first search traversal of each subtree. By the same 
reasoning as the previous case, we see that 
$$
2\|H_1\| > |V(H_1)\cap D_1|\cdot \frac{2}{k+3}\diam_k(G) + (|V(H_1)\cap 
D_1|-2)\ell_1,
$$
and
$$
2\|H_2\| > |V(H_2)\cap D_1|\cdot \frac{2}{k+3}\diam_k(G) + (|V(H_2)\cap 
D_1|-2)\ell_1.
$$
Combining these sums together, we see that 
$$
2\|H_1\| + 2\|H_2\|> (k+1)\cdot \frac{2}{k+3}\diam_k(G) + (k+1-4)\ell_1.
$$
Since $k\geq 5$, we have that 
$$
2\|H_1\| + 2\|H_2\|> \frac{2(k+1)}{k+3}\diam_k(G) + 2\ell_1.
$$
Hence,
\begin{align*}
\|H_1\| + \|H_2\|-\ell_1&>\frac{k+1}{k+3}\diam_k(G),
\end{align*}
which contradicts equation \eqref{heqn}.
So no such connected graph $G$ exists. 
\end{proof}

\section{Sharpness of Theorem \ref{MAIN}}\label{ub}

We now prove that this bound in Theorem~\ref{MAIN} is tight via a construction. 
Let $k\geq 5$ be an integer. We now outline the construction of a graph $G_k$ 
satisfying 
$$
\diam_k(G_k)=\frac{k+3}{k+1}\rad_k(G_k).
$$
Begin with a set of $k$ 
independent vertices, $D=\{d_1,d_2,\ldots,d_k\}$. Let $m=\lceil 
\frac{k+1}{2}\rceil$. Define $D_1=\{d_1,d_2,\ldots,d_{m}\}$ and 
$D_2=\{d_{m},d_{m+1},\ldots,d_k\}$. For each vertex $d_i\in D_1$, adjoin to each 
vertex in $D_1\setminus \{d_i\}$ a new vertex $a_i$. Let $A$ be the set these 
new vertices all such vertices. Similarly, for each vertex $d_j\in D_2$ define a 
new vertex $b_j$ to be a vertex with $N(b_u)=D_2\setminus\{d_j\}$. Let $B$ be 
the set of all such vertices. 
Finally, adjoin a new vertex $r$ to each vertex in $A\cup B$. This completes the 
construction of $G_k$. Figure~\ref{G5} and Figure~\ref{G6} illustrate the
graphs of $G_5$ and $G_6$, respectively. 

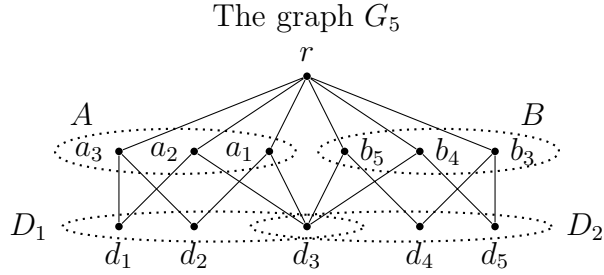
\begin{figure}[ht!]
\centering
\begin{tabular}{c}
The graph $G_5$\\
\begin{tikzpicture}[,node distance=30mm]
\newcommand\X{1};
\newcommand\Y{1};
\node[minimum size=1mm,inner sep=0pt] (r) at (0,0) [circle, 
fill=black,label=above:$r$] {};

\node[minimum size=1mm,inner sep=0pt] (a1) at (-2.5*\X,-1*\Y) [circle, 
fill=black,label=left:$a_3$] {};
\node[minimum size=1mm,inner sep=0pt] (a2) at (-1.5*\X,-1*\Y) [circle, 
fill=black,label=left:$a_2$] {};
\node[minimum size=1mm,inner sep=0pt] (a3) at (-.5*\X,-1*\Y) [circle, 
fill=black,label=left:$a_1$] {};

\node[minimum size=1mm,inner sep=0pt] (b4) at (.5*\X,-1*\Y) [circle, 
fill=black,label=right:$b_5$] {};
\node[minimum size=1mm,inner sep=0pt] (b5) at (1.5*\X,-1*\Y) [circle, 
fill=black,label=right:$b_4$] {};
\node[minimum size=1mm,inner sep=0pt] (b6) at (2.5*\X,-1*\Y) [circle, 
fill=black,label=right:$b_3$] {};

\node[minimum size=1mm,inner sep=0pt] (d1) at (-2.5*\X,-2*\Y) [circle, 
fill=black,label=below:$d_1$] {};
\node[minimum size=1mm,inner sep=0pt] (d2) at (-1.5*\X,-2*\Y) [circle, 
fill=black,label=below:$d_2$] {};
\node[minimum size=1mm,inner sep=0pt] (d3) at (0*\X,-2*\Y) [circle, 
fill=black,label=below:$d_3$] {};
\node[minimum size=1mm,inner sep=0pt] (d4) at (1.5*\X,-2*\Y) [circle, 
fill=black,label=below:$d_4$] {};
\node[minimum size=1mm,inner sep=0pt] (d5) at (2.5*\X,-2*\Y) [circle, 
fill=black,label=below:$d_5$] {};

\draw (r)--(a1) (r)--(a2) (r)--(a3) (r)--(b5) (r)--(b6) (r)--(b4);

\draw (a1)--(d1) (a1)--(d2);
\draw (a2)--(d1) (a2)--(d3);   
\draw (a3)--(d2) (a3)--(d3);

\draw (d3)--(b4) (d4)--(b4) (d3)--(b5) (d5)--(b5) (d4)--(b6) (d5)--(b6);

\draw[dotted,thick] (-1.75*\X,-\Y) ellipse (16mm and 3mm);
\draw[dotted,thick] (1.75*\X,-\Y) ellipse (16mm and 3mm);

\draw[dotted,thick] (-1.25*\X,-2*\Y) ellipse (20mm and 2mm);
\draw[dotted,thick] (1.25*\X,-2*\Y) ellipse (20mm and 2mm);

\node (q) at (-3.*\X,-.5*\Y) {$A$};
\node (q) at (3.*\X,-.5*\Y) {$B$};

\node (q) at (-3.7*\X,-2*\Y) {$D_1$};
\node (q) at (3.7*\X,-2*\Y) {$D_2$};
\end{tikzpicture}
\end{tabular}
\caption{The graph $G_5$. All vertices are drawn.}
\label{G5}
\end{figure}

\begin{figure}[ht!]
\centering
\begin{tabular}{c}
The graph $G_6$\\

\begin{tikzpicture}[,node distance=30mm]
\newcommand\X{1};
\newcommand\Y{1};
\node[minimum size=1mm,inner sep=0pt] (r) at (0,0) [circle, 
fill=black,label=above:$r$] {};

\node[minimum size=1mm,inner sep=0pt] (a1) at (-3*\X,-1*\Y) [circle, 
fill=black,label=left:$a_4$] {};
\node[minimum size=1mm,inner sep=0pt] (a2) at (-2*\X,-1*\Y) [circle, 
fill=black,fill=black,label=left:$a_3$]  {};
\node[minimum size=1mm,inner sep=0pt] (a3) at (-1*\X,-1*\Y) [circle, 
fill=black,label=left:$a_2$] {};
\node[minimum size=1mm,inner sep=0pt] (a4) at (0*\X,-1*\Y) [circle,
fill=black,label=left:$a_1$] 
{};

\node[minimum size=1mm,inner sep=0pt] (b5) at (1*\X,-1*\Y) [circle,
fill=black,label=right:$b_6$] 
{};
\node[minimum size=1mm,inner sep=0pt] (b6) at (2*\X,-1*\Y) [circle,
fill=black,label=right:$b_5$] 
{};
\node[minimum size=1mm,inner sep=0pt] (b7) at (3*\X,-1*\Y) [circle,
fill=black,label=right:$b_4$] 
{};

\node[minimum size=1mm,inner sep=0pt] (d1) at (-3*\X,-2*\Y) [circle, 
fill=black,label=below:$d_1$] {};
\node[minimum size=1mm,inner sep=0pt] (d2) at (-2*\X,-2*\Y) [circle, 
fill=black,label=below:$d_2$] {};
\node[minimum size=1mm,inner sep=0pt] (d3) at (-1*\X,-2*\Y) [circle, 
fill=black,label=below:$d_3$] {};
\node[minimum size=1mm,inner sep=0pt] (d4) at (0.5*\X,-2*\Y) [circle, 
fill=black,label=below:$d_4$] {};
\node[minimum size=1mm,inner sep=0pt] (d5) at (2*\X,-2*\Y) [circle, 
fill=black,label=below:$d_5$] {};
\node[minimum size=1mm,inner sep=0pt] (d6) at (3*\X,-2*\Y) [circle, 
fill=black,label=below:$d_6$] {};

\draw (r)--(a1) (r)--(a2) (r)--(a3) (r)--(a4) (r)--(b5) (r)--(b6) (r)--(b7);

\draw (a1)--(d1) (a1)--(d2) (a1)--(d3);
\draw (a2)--(d1) (a2)--(d2) (a2)--(d4);   
\draw (a3)--(d1) (a3)--(d4) (a3)--(d3);
\draw (a4)--(d4) (a4)--(d2) (a4)--(d3);

\draw (d4)--(b5) (d5)--(b5) (d4)--(b6) (d6)--(b6) (d5)--(b7) (d6)--(b7);

\draw[dotted,thick] (-1.75*\X,-\Y) ellipse (21mm and 3mm);
\draw[dotted,thick] (2.25*\X,-\Y) ellipse (16mm and 3mm);

\draw[dotted,thick] (-1.25*\X,-2*\Y) ellipse (25mm and 2mm);
\draw[dotted,thick] (1.75*\X,-2*\Y) ellipse (20mm and 2mm);

\node (q) at (-3.5*\X,-.5*\Y) {$A$};
\node (q) at (3.5*\X,-.5*\Y) {$B$};

\node (q) at (-4.2*\X,-2*\Y) {$D_1$};
\node (q) at (4.2*\X,-2*\Y) {$D_2$};

\end{tikzpicture}
\end{tabular}
\caption{The graph $G_6$. All vertices are drawn.}
\label{G6}
\end{figure}
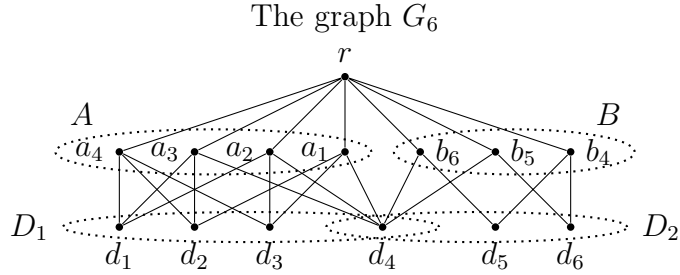

We now show that ${\diam_k(G_k)=k+3}$ and
${\rad_k(G_k)=k+1}$ via the Claims~\ref{claimr} and \ref{claimd}.

\begin{claim}\label{claimr}
In the graph $G_k$, we have $\e_k(r)=k+1$.
\end{claim}

\begin{proof}
We first show that $\e_k(r)\leq k+1$.
Suppose that $S\in \binom{V(G_k)}{k}$ with $r\in S$ such that $d(S)=e_k(r)$. Note that $|D\cap S|\leq k-1$ which implies that there exist $a,b\in A\cup B$ such that $(S\cap D) \subseteq (N(a)\cup N(b))$. So the Steiner tree of $S\cup \{a,b\}$ has order at most $k+2$ and size at most $k+1$. Since $S\subseteq (S\cup \{a,b\})$, we have $d(S)\leq k+1$. This implies that $\e_k(r)\leq k+1$. 

To show that $\e_k(r)\geq k+1$, consider the set 
$R=\{r,d_1,d_2,\ldots,d_{k-1}\}$.
Since $|R\cap D| > \lceil k/2 \rceil+1$,
any Steiner tree of $R$ must contain at least two elements of $A\cup B$.
So any Steiner tree of $R$ must have order $k+2$ and size $k+1$.
Hence, $e_k(r)\geq d(R) \geq k+1$. 
A Steiner tree of $S$ is illustrated on the right side of Figure~\ref{Gcheck}. 
\end{proof}

\begin{figure}[ht!]
\centering
\begin{tabular}{ll}
\begin{tikzpicture}[,node distance=30mm]
\newcommand\X{1};
\newcommand\Y{1};
\node[minimum size=1mm,inner sep=0pt] (r) at (0,0) [circle, 
fill=black,label=above:$r$] {};

\node[minimum size=1mm,inner sep=0pt] (a1) at (-2.5*\X,-1*\Y) [circle, 
fill=black,label= left:$a_3$] {};
\node[minimum size=1mm,inner sep=0pt] (a2) at (-1.5*\X,-1*\Y) [circle, 
fill=black,label=left:$a_2$] {};
\node[minimum size=1mm,inner sep=0pt] (a3) at (-.5*\X,-1*\Y) [circle, 
fill=black,label=left:$a_1$] {};

\node[minimum size=1mm,inner sep=0pt] (b4) at (.5*\X,-1*\Y) [circle, 
fill=black,label=right:$b_5$] {};
\node[minimum size=1mm,inner sep=0pt] (b5) at (1.5*\X,-1*\Y) [circle, 
fill=black,label=right:$b_4$] {};
\node[minimum size=1mm,inner sep=0pt] (b6) at (2.5*\X,-1*\Y) [circle, 
fill=black,label= right:$b_3$] {};

\node[minimum size=1mm,inner sep=0pt] (d1) at (-2.5*\X,-2*\Y) [circle, 
fill=black,label=below:$d_1$] {};
\node[minimum size=1mm,inner sep=0pt] (d2) at (-1.5*\X,-2*\Y) [circle, 
fill=black,label=below:$d_2$] {};
\node[minimum size=1mm,inner sep=0pt] (d3) at (0*\X,-2*\Y) [circle, 
fill=black,label=below:$d_3$] {};
\node[minimum size=1mm,inner sep=0pt] (d4) at (1.5*\X,-2*\Y) [circle, 
fill=black,label=below:$d_4$] {};
\node[minimum size=1mm,inner sep=0pt] (d5) at (2.5*\X,-2*\Y) [circle, 
fill=black,label=below:$d_5$] {};

\draw (r)--(a1) (r)--(b4);

\draw (r)--(a2) (r)--(a3) (r)--(b5) (r)--(b6);

\draw[ultra thick] (a1)--(d1) (a1)--(d2);
\draw (a2)--(d1) (a2)--(d3);   
\draw[ultra thick] (a3)--(d2) (a3)--(d3);

\draw (d3)--(b4) (d4)--(b4) (d3)--(b5) (d5)--(b5);

\draw[ultra thick] (d5)--(b6) (d4)--(b6);
\draw[ultra thick] (d3)--(b4) (d4)--(b4);




\end{tikzpicture}
\hspace{.25in}
&
\begin{tikzpicture}[,node distance=30mm]
\newcommand\X{1};
\newcommand\Y{1};
\node[minimum size=1mm,inner sep=0pt] (r) at (0,0) [circle, 
fill=black,label=above:$r$] {};

\node[minimum size=1mm,inner sep=0pt] (a1) at (-2.5*\X,-1*\Y) [circle, 
fill=black,label= left:$a_3$] {};
\node[minimum size=1mm,inner sep=0pt] (a2) at (-1.5*\X,-1*\Y) [circle, 
fill=black,label=left:$a_2$] {};
\node[minimum size=1mm,inner sep=0pt] (a3) at (-.5*\X,-1*\Y) [circle, 
fill=black,label=left:$a_1$] {};

\node[minimum size=1mm,inner sep=0pt] (b4) at (.5*\X,-1*\Y) [circle, 
fill=black,label=right:$b_5$] {};
\node[minimum size=1mm,inner sep=0pt] (b5) at (1.5*\X,-1*\Y) [circle, 
fill=black,label=right:$b_4$] {};
\node[minimum size=1mm,inner sep=0pt] (b6) at (2.5*\X,-1*\Y) [circle, 
fill=black,label= right:$b_3$] {};

\node[minimum size=1mm,inner sep=0pt] (d1) at (-2.5*\X,-2*\Y) [circle, 
fill=black,label=below:$d_1$] {};
\node[minimum size=1mm,inner sep=0pt] (d2) at (-1.5*\X,-2*\Y) [circle, 
fill=black,label=below:$d_2$] {};
\node[minimum size=1mm,inner sep=0pt] (d3) at (0*\X,-2*\Y) [circle, 
fill=black,label=below:$d_3$] {};
\node[minimum size=1mm,inner sep=0pt] (d4) at (1.5*\X,-2*\Y) [circle, 
fill=black,label=below:$d_4$] {};
\node[minimum size=1mm,inner sep=0pt] (d5) at (2.5*\X,-2*\Y) [circle, 
fill=black,label=below:$d_5$] {};

\draw[ultra thick] (r)--(a1) (r)--(b4);

\draw (r)--(a2) (r)--(a3) (r)--(b5) (r)--(b6);

\draw[ultra thick] (a1)--(d1) (a1)--(d2);
\draw (a2)--(d1) (a2)--(d3);   
\draw (a3)--(d2) (a3)--(d3);

\draw (d3)--(b4) (d4)--(b4) (d3)--(b5) (d5)--(b5) (d4)--(b6) (d5)--(b6);

\draw[ultra thick] (d3)--(b4) (d4)--(b4);




\end{tikzpicture}\\
A Steiner tree for $D$ realizing
&
A Steiner tree for $R=\{r,d_1,d_2,d_3,d_4\}$\\
$d(D)=\diam_5(G_5)=8$.
&
realizing $d(R)=\rad_5(G_5)=6$.
\end{tabular}
\caption{Steiner trees in the graph $G_5$ for $D$ as in Definition~\ref{defn} and $R$. The Steiner trees are in bold.}
\label{Gcheck}
\end{figure}
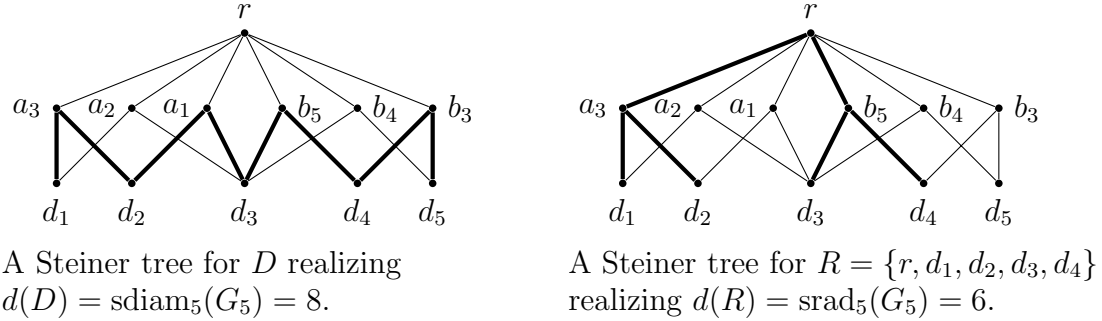

Recall that $\rad(G_k)\leq \e_k(r)$,
which implies that $\rad_k(G_k)\leq k+1$.
So by Theorem~\ref{MAIN} we need only find a vertex $v\in V(G_k)$
with $e_k(v)\geq k+3$ to prove that $\rad_k(G_k)=k+1$ and
$\diam_k(G_k)=k+3$. 
This will show that Theorem~\ref{MAIN} is tight. 
We now consider the vertex $d_1$. 

\begin{claim}\label{claimd}
In the graph $G_k$, $\e_k(d_1)= k+3$.
\end{claim}

\begin{proof}
Note that $D\in \binom{V(G_k)}{k}$ and $d_1\in D$ so $d(D)\leq \e_k(d_1)$. 
Let $T$ be a Steiner tree of $D$. 
Since $|D|=k$, we have $|V(T)\cap (A\cup B)|\geq 3$. 
If $|V(T)\cap (A\cup B)|=4$, then $|T|\geq k+4$ and $\|T\|\geq k+3$. 

We now suppose $|V(T)\cap (A\cup B)| = 3$. 
This implies that either $|V(T)\cap A|=2$ or $|V(T)\cap B|=2$.
Without loss of generality, we 
suppose the former. This implies that $V(T)\cap B=\{b_m\}$. Since $T$ is 
connected, $T$ contains a path from $d_m$ to $b_m$. Such a path either
contains 
a second element from $B$ or the vertex $r$. Since $|V(T)\cap B|=1$, we
must 
have that $T$ contains $r$. Hence, $|T|\geq k+4$ and $\|T\|\geq k+3$. 
We have now shown that $\e_k(d_1)\geq k+3$. 
\end{proof}

We have shown in Claim~\ref{claimr} $\e_k(r)\leq k+1$ and in Claim~\ref{claimd} that $\e_k(d_1)\geq k+3$.
With Claims~\ref{claimd} and \ref{claimr} in hand, we have 
$\rad_k(G_k)\leq k+1$ and $\diam_k(G_k)\geq k+3$.
Applying Theorem~\ref{MAIN}, 
we infer the following proposition, which shows the bound in Theorem~\ref{MAIN} is tight. 

\begin{prop}
For $k\geq 5$ the graph $G_k$ satisfies $\diam_k(G_k)=k+3$ and $\rad_k(G_k)=k+1$.
\end{prop}

\section{Proof of Conjecture \ref{bad} for $k=4$}\label{ub4}

We now consider the case of $k=4$. 
Following the methods in \cite{HEN90}, 
we suppose towards a contradiction that there exists 
a connected graph $G$ satisfying 
$$
\diam_4(G) > \frac{10}{7}\rad_4(G).
$$
Suppose that $D=\{v_1,v_2,v_3,v_4\}$ is a set of vertices in $G$ such that 
$d(D)=\diam_4(G)$ and suppose $v_0\in C_4(G)$. For $1\leq i\leq 4$, 
define $D_i$, $T_i$, 
$T_i'$, and $\ell_i$ as in Definition~\ref{defn}. Again, we assume that 
$\ell_{1}\leq \ell_j$ for $j\geq 2$. 

\noindent \textbf{Remark}:
In \cite{HEN90}, the authors show that such a counterexample must satisfy
the following two properties:
\begin{enumerate}
\item For each $1\leq i\leq 4$, the tree $T_i$ must have exactly four leaves.
We label the leaves $v_0,u_1^{(i)},u_2^{(i)},u_3^{(i)},u_4^{(i)}$ so that $v_0$
and $u_3^{(i)}$ share a nearest vertex $s_i$ of degree at least 3 
while $u_1^{(i)}$ and $u_2^{(i)}$ share a nearest vertex $t_i$ 
of degree at least 3. It is possible that $s_i=t_i$.
This implies that each tree $T_i$ is of the form shown in Figure~\ref{G4w}
where
$$
\begin{array}{cc}
a:=d_{T_1}(u_1^{(i)},t)\hspace{.5in}& b:=d_{T_1}(u_2^{(i)},t)\\
c:=d_{T_1}(u_3^{(i)},s)\hspace{.5in}& d:=d_{T_1}(s,t).
\end{array}
$$
An illustration of each tree $T_i$ is included in Figure~\ref{G4w}.
\item The length $\ell_1< \dfrac{1}{10}\diam_4(G)$. 
We provide a short justification of this fact in the following argument. 
\end{enumerate}
After proving these claims, the authors
define $T_i''$ to be the subtree of $T_i$ 
obtained by deleting the vertices in the $u_3^{(i)}-s_i$ path except for $s$.
Figure~\ref{G4w} illustrates the difference between $T_i$ and $T_i''$.
We will not use $T''$ in this section but will examine it in the following section. 
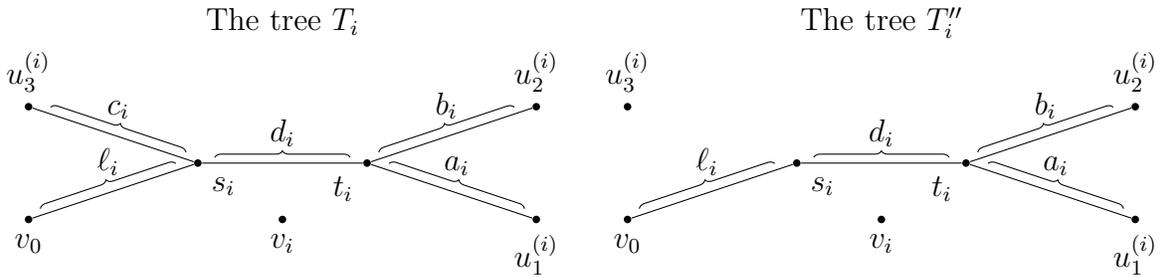
\begin{figure}[ht!]
\centering
\begin{tabular}{cc}
The tree $T_i$ 
&
The tree $T_i''$\\
\begin{tikzpicture}[scale=.75]
\newcommand\X{3};
\newcommand\Y{1};
\node[minimum size=1mm,inner sep=0pt] (r) at (0,0) [circle, 
fill=black,label=below:$v_0$] {};
\node[minimum size=1mm,inner sep=0pt] (u1) at (3*\X,0) [circle, 
fill=black,label=below:$u_1^{(i)}$] {};
\node[minimum size=1mm,inner sep=0pt] (u2) at (3*\X,2*\Y) [circle, 
fill=black,label=above:$u_2^{(i)}$] {};
\node[minimum size=1mm,inner sep=0pt] (u3) at (0,2*\Y) [circle, 
fill=black,label=above:$u_3^{(i)}$] {};
\node[minimum size=1mm,inner sep=0pt] (s) at (1*\X,1*\Y) [circle, 
fill=black,label=below right:$s_i$] {};
\node[minimum size=1mm,inner sep=0pt] (t) at (2*\X,1*\Y) [circle, 
fill=black,label=below left:$t_i$] {};
\node[minimum size=1mm,inner sep=0pt] (v1) at (1.5*\X,0*\Y) [circle, 
fill=black,label=below:$v_i$] {};
\draw  (r) edge (s) (s) edge (u3);
\draw (u2) edge (t) (t) edge (u1);
\draw (s) edge (t);
\draw[decorate, decoration={brace}, yshift=.5ex,xshift=-.5ex]  (.1*\X,.1*\Y) -- 
node[above] {$\ell_i$}  (\X-.1*\X,.9*\Y);
\draw[decorate, decoration={brace}, yshift=.5ex,xshift=.5ex]  (.1*\X,1.9*\Y) -- 
node[above] {$c_i$}  (.9*\X,1.1*\Y);
\draw[decorate, decoration={brace}, yshift=.5ex,xshift=-.5ex]  (2.1*\X,1.1*\Y) 
-- node[above] {$b_i$}  (2.9*\X,1.9*\Y);
\draw[decorate, decoration={brace}, yshift=.5ex,xshift=.5ex]  
(2*\X+.1*\X,\Y-.1*\Y) -- node[above] {$a_i$}  (3*\X-.1*\X,.1*\Y);
\draw[decorate, decoration={brace}, yshift=0ex]  (1.1*\X,1.1*\Y) -- node[above] 
{$d_i$}  (1.9*\X,1.1*\Y);
\end{tikzpicture}
&
\begin{tikzpicture}[scale=.75]
\newcommand\X{3};
\newcommand\Y{1};
\node[minimum size=1mm,inner sep=0pt] (r) at (0,0) [circle, 
fill=black,label=below:$v_0$] {};
\node[minimum size=1mm,inner sep=0pt] (u1) at (3*\X,0) [circle, 
fill=black,label=below:$u_1^{(i)}$] {};
\node[minimum size=1mm,inner sep=0pt] (u2) at (3*\X,2*\Y) [circle, 
fill=black,label=above:$u_2^{(i)}$] {};
\node[minimum size=1mm,inner sep=0pt] (u3) at (0,2*\Y) [circle, 
fill=black,label=above:$u_3^{(i)}$] {};
\node[minimum size=1mm,inner sep=0pt] (s) at (1*\X,1*\Y) [circle, 
fill=black,label=below right:$s_i$] {};
\node[minimum size=1mm,inner sep=0pt] (t) at (2*\X,1*\Y) [circle, 
fill=black,label=below left:$t_i$] {};
\node[minimum size=1mm,inner sep=0pt] (v1) at (1.5*\X,0*\Y) [circle, 
fill=black,label=below:$v_i$] {};
\draw  (r) edge (s);
\draw (u2) edge (t) (t) edge (u1);
\draw (s) edge (t);
\draw[decorate, decoration={brace}, yshift=.5ex,xshift=-.5ex]  (.1*\X,.1*\Y) -- 
node[above] {$\ell_i$}  (\X-.1*\X,.9*\Y);
\draw[decorate, decoration={brace}, yshift=.5ex,xshift=-.5ex]  (2.1*\X,1.1*\Y) 
-- node[above] {$b_i$}  (2.9*\X,1.9*\Y);
\draw[decorate, decoration={brace}, yshift=.5ex,xshift=.5ex]  
(2*\X+.1*\X,\Y-.1*\Y) -- node[above] {$a_i$}  (3*\X-.1*\X,.1*\Y);
\draw[decorate, decoration={brace}, yshift=0ex]  (1.1*\X,1.1*\Y) -- node[above] 
{$d_i$}  (1.9*\X,1.1*\Y);
\end{tikzpicture}
\end{tabular}
\caption{The trees $T_i$ and $T_i''$. Vertices of degree 2 are not drawn.}
\label{G4w}
\end{figure}

In reference to tree $T_1$, we consider the sum 
$$
2(\ell_1+a_1+b_1+c_1+d_1)=(\ell_1+c_1)+(c_1+d_1+b_1)+(a_1+b_1)+(\ell_1+d_1+a_1).
$$
By Corollary~\ref{distlemobs}, the left hand side is bounded below by
$$
(\ell_1+c_1)+(c+d_1+b_1)+(a_1+b_1)+(\ell_1+d_1+a_1)>\frac{12}{10}\diam_4(G)+2\ell_1,
$$
while, as in the previous case, by equation \eqref{4radsm}, we have 
that the right hand side is bounded below by
$$
2(\ell_1+a_1+b_1+c_1+d_1)=2\|T_i\|<\frac{14}{10}\diam_4(G).
$$
Combining these inequalities together, we have that 
$$
\frac{12}{10}\diam_4(G)+2\ell_1<\frac{14}{10}\diam_4(G),
$$
which implies that 
\begin{equation}
\label{4ellsm}
\ell_1<\frac{1}{10}\diam_4(G).
\end{equation} Alternatively, we may consider the sum 
$$
2\ell_1+2(\ell_1+a_1+b_1+c_1+d_1)=(\ell_1+d_1+b_1)+(\ell_1+d_1+a_1)+2(\ell_1+c_1)+(a_1+b_1).
$$
Applying Corollary~\ref{distlemobs}, we see that
\begin{align*}
(\ell_1+d_1+b_1)+(\ell_1+d_1+a_1)+2(\ell_1+c_1)+(a_1+b_1)>\frac{15}{10}\diam_4(G)+\ell_1. 
\end{align*}
But by equation \eqref{4radsm}, we have that 
$$
2\ell_1+2(\ell_1+a_1+b_1+c_1+d_1)<\frac{14}{10}\diam_4(G)+2\ell_1. 
$$
Combining these inequalities together, we see that 
$$
\frac{15}{10}\diam_4(G)+\ell_1<\frac{14}{10}\diam_4(G)+2\ell_1, 
$$
which implies that $\ell_1>\frac{1}{10}\diam_4(G)$, a contradiction of 
equation \eqref{4ellsm}. Hence, no such graph $G$ exists
and the result is proven. 

\section*{Examining a previous proof}

We now examine to the proof for $k=4$ presented in \cite{HEN90}.
The authors of \cite{HEN90} show correctly that if $G$ is a graph with 
$\diam_4(G)>\dfrac{10}{7}\rad_4(G)$, then, in reference to 
Definition~\ref{defn}, $G$ must satisfy $\ell_1<\dfrac{1}{10}\diam_4(G)$
and $d(u,v)>\dfrac{3}{10}\diam_4(G)$ for any $u,v\in D\cup\{v_0\}$. 

In their displayed equations 8, 9, and 10 
of Case 2 of Theorem~3 in \cite{HEN90},
 the authors make the following claims:
\begin{claim}[See \cite{HEN90}, Case 2 of Theorem 3]\label{wrong}
If $G$ is a graph satisfying $\ell_1<\dfrac{1}{10}\diam_4(G)$ and 
$d(u,v)>\dfrac{3}{10}\diam_4(G)$ for any $u,v\in D\cup\{v_0\}$.
In reference to Figure~\ref{G4w}, 
\begin{align*}
\|T_{2}''\|+a_1+b_1&\geq \diam_4(G)\\
\|T_{2}''\|+a_1+d_1+\ell_1&\geq \diam_4(G)\\
\|T_{3}''\|+b_1+d_1+\ell_1&\geq \diam_4(G).
\end{align*}
\end{claim}
The first of these claims is violated by the graph $H$ illustrated in
Figure~\ref{dada}.
We obtain $H$ by taking a the complete bipartite graph $K_{4,4}$ with bipartite 
sets $U=\{u_1,u_2,u_3,u_4\}$ and $V=\{v_1,v_2,v_3,v_4\}$ and deleting the 
matching $\{u_1v_4,u_2v_3,u_3v_2,u_4v_1\}$. 
We then adjoin a new vertex $v_0$ to 
each vertex in $U$. 
Finally, we subdivide the edge $v_0u_4$ once 
and each edge of the form $uv$ where $u\in U$ and $v\in V$ 
we subdivide five times. 
This gives the following distances in $H$:
$$
d(v_0,u_i)=1, 1\leq i\leq 3; \ d(v_0,u_4)=2;\text{ and }
d(u_i,v_j)=6\text{ where } i\neq 5-j.
$$

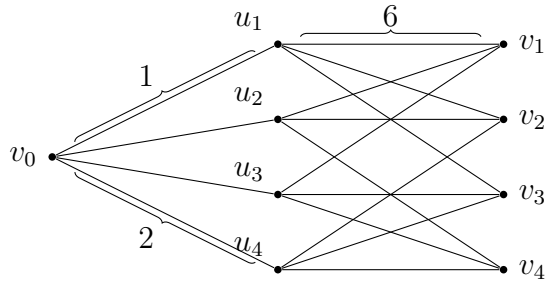
\begin{figure}[ht!]
\centering
\begin{tikzpicture}
\newcommand\X{3};
\newcommand\Y{1};
\node[minimum size=1mm,inner sep=0pt] (r) at (0,0) [circle, 
fill=black,label=left:$v_0$] {};

\node[minimum size=1mm,inner sep=0pt] (u1) at (1*\X,1.5*\Y) [circle, fill=black,label=above left: $u_1$] 
{};
\node[minimum size=1mm,inner sep=0pt] (u2) at (1*\X,.5*\Y) [circle, fill=black,label=above left: $u_2$] 
{};
\node[minimum size=1mm,inner sep=0pt] (u3) at (1*\X,-.5*\Y) [circle, fill=black,label=above left: $u_3$] 
{};
\node[minimum size=1mm,inner sep=0pt] (u4) at (1*\X,-1.5*\Y) [circle, 
fill=black,label=above left: $u_4$] {};

\node[minimum size=1mm,inner sep=0pt] (v1) at (2*\X,1.5*\Y) [circle, 
fill=black,label=right: $v_1$] {};
\node[minimum size=1mm,inner sep=0pt] (v2) at (2*\X,.5*\Y) [circle, 
fill=black,label=right: $v_2$] {};
\node[minimum size=1mm,inner sep=0pt] (v3) at (2*\X,-.5*\Y) [circle, 
fill=black,label=right: $v_3$] {};
\node[minimum size=1mm,inner sep=0pt] (v4) at (2*\X,-1.5*\Y) [circle, 
fill=black,label=right: $v_4$] {};

\draw (u1) edge (v1) (u1) edge (v2) (u1) edge (v3);
\draw (u2) edge (v1) (u2) edge (v2) (u2) edge (v4);
\draw (u3) edge (v1) (u3) edge (v3) (u3) edge (v4);
\draw (u4) edge (v2) (u4) edge (v3) (u4) edge (v4);

\draw (r) edge (u1) (r) edge (u2) (r) edge (u3) (r) edge (u4);

\draw[decorate, decoration={brace}, yshift=0ex]  (1.1*\X,1.6*\Y) -- node[above] 
{$6$}  (1.9*\X,1.6*\Y);
\draw[decorate, decoration={brace}]  (.1*\X,.2*\Y) -- node[above left] 
{$1$}  (.9*\X,1.4*\Y);
\draw[decorate, decoration={brace}]  (.9*\X,-1.4*\Y) -- node[below left] 
{$2$}  (.1*\X,-.2*\Y);
\end{tikzpicture}
\caption{The graph $H$. Vertices of degree 2 are not drawn.}
\label{dada}
\end{figure}
In reference to the tree $T_i$ in Figure~\ref{G4w} as a subgraph of the graph
$H$ in Figure~\ref{dada}, we have that in the tree $T_1$
satisfies the following: $a_1=6$, $b_1=6$, $c_1=6$, $d_1=0$, 
and $\ell_1=2$. 
We have that 
$$
\rad_4(H)=\e_4(v_0)=3(6)+2=20
$$
while 
$$
\diam_4(H)=d(v_1,v_2,v_3,v_4)=4(6)+2(1)=26.
$$
In reference to the graph $H$, from the first displayed equation in
Claim~\ref{wrong}, we have that $\|T_2''\|=2(6)+1=13$,
while within the tree $T_1$, we have $a_1=6$, and $b_1=6$. So 
$$
\|T_2''\|+a_1+b_1=25<\diam_4(H),
$$
which contradicts the first inequality in Claim~\ref{wrong}.
The remaining inequalities in Claim~\ref{wrong} can be violated graphs
similar to $H$. 

\section{Completing the proof for $k=4$ using arguments from \cite{HEN90}}
\label{complete}

For completeness, we show that if $G$ is a graph satisfying 
$\diam_4(G)>\dfrac{10}{7}\rad_4(G)$, 
then each subtree $T_i$ as in Definition~\ref{defn} contains four leaves. 
This section restates the arguments correctly presented in \cite{HEN90} 
using the notation of this paper.
\begin{proof}
Suppose $G$ is a graph satisfying 
$$
\diam_4(G) > \frac{10}{7}\rad_4(G).
$$
This implies that 
\begin{equation}
\label{4radsm}
\rad_4(G)<\frac{7}{10}\diam_4(G).
\end{equation}
Suppose that $D=\{v_1,v_2,v_3,v_4\}$ is a set of vertices in $G$ such that 
$d(D)=\diam_4(G)$ and $v_0\in C_4(G)$. For $1\leq i\leq 4$, 
define $D_i$, $T_i$, 
$T_i'$, and $\ell_i$ as in Definition~\ref{defn}. Again, we assume that 
$\ell_{1}\leq \ell_j$ for $j\geq 2$. 

We consider the cases where $T_1$ is a path or a subdivision of the star 
on three vertices.
First, suppose that $T_1$ is a path. Relabel the elements of $D_1$ as 
$u_1,u_2,u_3$ and $u_4$ so that the tree $T_1$ is a concatenation of paths 
$u_1-u_2-u_3-u_4$. See Figure~\ref{gpath} for an illustration of this situation. 

\begin{figure}[ht!]
\centering
\begin{tabular}{c}
\begin{tikzpicture}
\newcommand\X{4};
\newcommand\Y{1};
\node[minimum size=1mm,inner sep=0pt] (u0) at (0,0*\Y) [circle, 
fill=black,label=below:$u_1$] {};
\node[minimum size=1mm,inner sep=0pt] (u1) at (1*\X,0*\Y) [circle, 
fill=black,label=below:$u_2$] {};
\node[minimum size=1mm,inner sep=0pt] (u2) at (2*\X,0*\Y) [circle, 
fill=black,label=below:$u_3$] {};
\node[minimum size=1mm,inner sep=0pt] (u3) at (3*\X,0*\Y) [circle, 
fill=black,label=below:$u_4$] {};
\node[minimum size=1mm,inner sep=0pt] (v1) at (1.5*\X,-1*\Y) [circle, 
fill=black,label=below:$v_1$] {};

\draw  (u0)--(u1)--(u2)--(u3);
\draw[decorate, decoration={brace}] ([xshift = .5ex,yshift=.5ex]u0.center) -- 
([xshift = -.5ex,yshift=.5ex]u1.center) node[midway, above] {$>\frac{3}{10}\diam_4(G)$};
\draw[decorate, decoration={brace}] ([xshift = .5ex,yshift=.5ex]u1.center) -- 
([xshift = -.5ex,yshift=.5ex]u2.center) node[midway, above] {$>\frac{3}{10}\diam_4(G)$};
\draw[decorate, decoration={brace}] ([xshift = .5ex,yshift=.5ex]u2.center) -- 
([xshift = -.5ex,yshift=.5ex]u3.center) node[midway, above] {$>\frac{3}{10}\diam_4(G)$};
\end{tikzpicture}
\end{tabular}
\caption{The tree $T_1$ as a path. Vertices of degree two not in 
$D_1$ are not drawn.}
\label{gpath}
\end{figure}
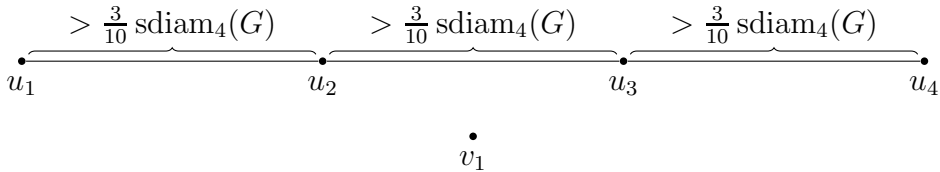

Now, $T_1$ is composed of three paths between elements of $D_i$. 
By Corollary~\ref{distlemobs}, each of these paths has length at least 
$\dfrac{3}{10}\diam_4(G)$. So 
$$
\rad_4(G)\geq ||T_1||>\dfrac{9}{10}\diam_4(G),
$$ which contradicts equation \eqref{4radsm}. 

Next, we suppose $T_1$ has exactly three leaves. Label 
them as $u_1,u_2$, and $u_3$. Let $u_4$ be the element of $D_1$ which is an 
interior vertex of $T_1$ and let $s$ be the vertex of degree 3 in $T_1$. It is 
possible that $s=u_4$. Without loss of generality, suppose that $u_4$ lies on the 
$s-u_3$ path in $T_1$. Define the following distances as illustrated in Figure~\ref{g4small}. 
$$
\begin{array}{cc}

a:=d_{T_1}(u_1,s)
&
b:=d_{T_1}(u_2,s)\\
c:=d_{T_1}(u_3,u_4)
&
d:=d_{T_1}(u_4,s).
\end{array}
$$

\begin{figure}[ht!]
\centering
\begin{tabular}{c}
\begin{tikzpicture}
\newcommand\X{4};
\newcommand\Y{1};
\node[minimum size=1mm,inner sep=0pt] (v2) at (0,0) [circle, 
fill=black,label=below:$u_1$] {};

\node[minimum size=1mm,inner sep=0pt] (v0) at (0,2*\Y) [circle, 
fill=black,label=above:$u_2$] {};
\node[minimum size=1mm,inner sep=0pt] (s) at (1*\X,1*\Y) [circle, 
fill=black,label=above:$s$] {};
\node[minimum size=1mm,inner sep=0pt] (v3) at (2*\X,1*\Y) [circle, 
fill=black,label=right:$u_3$] {};
\node[minimum size=1mm,inner sep=0pt] (v1) at (1.5*\X,0*\Y) [circle, 
fill=black,label=below:$v_1$] {};
\node[minimum size=1mm,inner sep=0pt] (v4) at (1.5*\X,1*\Y) [circle, 
fill=black,label=below:$u_4$] {};
\draw  (s) edge (v3) (s) edge (v0);
\draw[decorate, decoration={brace}] ([xshift = 1ex,yshift=0ex]v0.center) -- 
([xshift = -.5ex,yshift=.5ex]s.center) node[midway, above] {$b$};
\draw[decorate, decoration={brace}] ([xshift = .5ex,yshift=.5ex]v2.center) -- 
([xshift = -1ex,yshift=0ex]s.center) node[midway, above] {$a$};
\draw[decorate, decoration={brace}] ([xshift = .5ex,yshift=.5ex]s.center) -- 
([xshift = -.5ex,yshift=.5ex]v4.center) node[midway, above] {$d$};
\draw[decorate, decoration={brace}] ([xshift = .5ex,yshift=.5ex]v4.center) -- 
([xshift = -.5ex,yshift=.5ex]v3.center) node[midway, above] {$c$};
\draw (v2) edge (s);
\end{tikzpicture}
\end{tabular}
\caption{The tree $T_1$ with only three leaves. Vertices of degree two not in 
$D_1$ are not drawn.}
\label{g4small}
\end{figure}
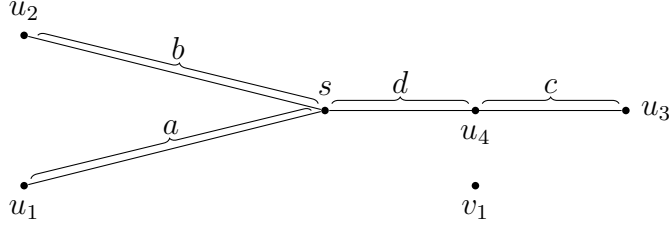

Consider the following sum:
$$
(a+b)+(a+d)+(b+d)+2c=2a+2b+2c+2d.
$$
The right hand side of this equation counts each edge of $T_1$ twice. Hence, by equation \eqref{4radsm},
\begin{equation}
\label{eq4}
2a+2b+2c+2d=2\|T_1\|\leq 2\rad_4(G)<\frac{14}{10}\diam_4(G).
\end{equation}
But Corollary~\ref{distlemobs} implies that the left hand side of the equation is bounded below by
\begin{align*}
(a+b)+(a+d)+(b+d)+2c
&\geq 5\cdot \min\{ d_{T_1}(u_i,u_j):1\leq i\neq j\leq 4 \}\\
&>5\cdot \frac{3}{10}\diam_4(G)\\
&=\frac{15}{10}\diam_4(G),
\end{align*}
which contradicts equation \eqref{eq4}.
The remainder of the proof for $k=4$ is considered in Section~\ref{ub4}.
\end{proof}

\bibliography{thesis}{}
\bibliographystyle{plain}
\end{document}